\documentclass[10pt,reqno]{amsart}
\usepackage{amssymb,color}
\usepackage[all]{xy}
\usepackage[lite,initials]{amsrefs}

\newenvironment{textequation}
{\begin{equation}\begin{minipage}{.87\textwidth}} {\end{minipage}
\end{equation}}

\DeclareMathOperator{\st}{st}
\DeclareMathOperator{\PC}{PC}
\DeclareMathOperator{\D}{D}
\DeclareMathOperator{\Dp}{D_P}

\DeclareMathOperator{\AC}{AC}
\DeclareMathOperator{\Conn}{Conn}

\DeclareMathOperator{\Ext}{Ext}
\DeclareMathOperator{\supp}{supp}
\DeclareMathOperator{\PR}{PR}
\DeclareMathOperator{\CIVP}{CIVP}
\DeclareMathOperator{\SCIVP}{SCIVP}
\DeclareMathOperator{\WCIVP}{WCIVP}
\DeclareMathOperator{\bor}{\mathcal{B}or}
\DeclareMathOperator{\bd}{bd}
\DeclareMathOperator{\CC}{C}
\DeclareMathOperator{\ES}{ES}
\DeclareMathOperator{\SES}{SES}
\DeclareMathOperator{\PES}{PES}
\DeclareMathOperator{\J}{J}

\DeclareMathOperator{\HL}{\mathcal{HL}}
\DeclareMathOperator{\mL}{m\mathcal{L}}
\DeclareMathOperator{\LL}{\mathcal L}

\newcommand{\real}{{\mathbb R}}

\newcommand{\integer}{{\mathbb Z}}

\newcommand{\A}{\mathcal A}
\newcommand{\B}{\mathcal B}

\newcommand{\F}{\mathcal F}
\newcommand{\G}{\mathcal G}
\newcommand{\e}{{\varepsilon}}
\newcommand{\charf}[1]{\mbox{\raise.48ex\hbox{$\chi$}$_{#1}$}}
\newcommand{\continuum}{\mathfrak c}

\newcommand{\la}{\langle}
\newcommand{\ra}{\rangle}
\newtheorem{theorem}{Theorem}[section]

\newtheorem{proposition}[theorem]{Proposition}
\newtheorem{lemma}[theorem]{Lemma}

\theoremstyle{definition}
\newtheorem{problem}[theorem]{Problem}

\newtheorem{remark}[theorem]{Remark}
\newtheorem{definition}[theorem]{Definition}
\begin{document}

\title[Lineability, spaceability, and additivity]{Lineability, spaceability, and additivity cardinals for Darboux-like functions}

\date{}

\author[Ciesielski]{Krzysztof Chris Ciesielski}
\address{Department of Mathematics,\newline \indent West Virginia University, Morgantown,\newline \indent WV 26506-6310, USA.\newline \indent
\textsc{ and }\newline \indent  \noindent Department of Radiology, MIPG,\newline \indent University of Pennsylvania,\newline \indent
Blockley Hall - 4th Floor, 423\newline \indent Guardian Drive,\newline \indent Philadelphia, PA 19104-6021, USA.}
\email{KCies@math.wvu.edu}

\author[G\'amez]{Jos\'{e} L.\ G\'{a}mez-Merino}
\address{Departamento de An\'{a}lisis Matem\'{a}tico,\newline\indent Facultad de Ciencias Matem\'{a}ticas, \newline\indent Plaza de Ciencias 3, \newline\indent Universidad Complutense de Madrid,\newline\indent Madrid, 28040, Spain.}
\email{jlgamez@mat.ucm.es}

\author[Pellegrino]{Daniel Pellegrino}
\address{Departamento de Matem\'{a}tica, \newline \indent Universidade Federal da Para\'{i}ba, \newline \indent 58.051-900 - Jo\~{a}o Pessoa, Brazil.}
\email{dmpellegrino@gmail.com and pellegrino@pq.cnpq.br}

\author[Seoane]{Juan B. Seoane-Sep\'{u}lveda}
\address{Departamento de An\'{a}lisis Matem\'{a}tico,\newline\indent Facultad de Ciencias Matem\'{a}ticas, \newline\indent Plaza de Ciencias 3, \newline\indent Universidad Complutense de Madrid,\newline\indent Madrid, 28040, Spain.}
\email{jseoane@mat.ucm.es}


\begin{abstract}
We introduce the concept of {\em maximal lineability cardinal number},
$\mL(M)$, of a subset $M$ of a topological vector space and study its relation to the cardinal numbers
known as: additivity $A(M)$, homogeneous lineability $\HL(M)$, and lineability $\LL(M)$ of $M$.
In particular, we will describe, in terms of $\LL$, the lineability and spaceability of the
families of the following Darboux-like functions on $\real^n$, $n\ge 1$:
extendable, Jones, and almost continuous functions.
\end{abstract}

\maketitle

\section{Preliminaries and background}

The work presented here is a contribution to a recent ongoing research concerning the following general question: \emph{For an arbitrary subset $M$ of a vector space $W$, how big can be a vector subspace $V$ contained in $M\cup\{0\}$?} The current state of knowledge concerning this problem is described in the very recent survey article \cite{Surv}. So far, the term \emph{big}\/ in the question was understood as a cardinality of a basis of $V$; however, some other measures of bigness (i.e., in  a category sense) can also be considered.

Following \cites{arongurariyseoane2005,seoane2006} (see, also, \cite{enflogurariyseoane2012}), given a cardinal number $\mu$ we say that $M\subset W$ is \emph{$\mu$-lineable} if $M \cup \{0\}$ contains a vector subspace
$V$ of the dimension $\dim( V)=\mu$. Consider the following \emph{lineability}\/ cardinal number (see \cite{bgLAA}):
    \[
    \LL(M)=\min\{\kappa\colon
    \text{$M \cup \{0\}$ contains no vector space of dimension $\kappa$}\}.
    \]
Notice that $M\subset W$ is $\mu$-lineable if, and only if, $\mu<\LL(M)$. In particular, $\mu$ is the maximal dimension of a subspace of $M \cup \{0\}$ if, and only if, $\LL(M)=\nolinebreak\mu^+$. The number $\LL(M)$ need not be a cardinal successor
(see, e.g.,~\cite{arongurariyseoane2005}); thus, the maximal dimension of a subspace of $M \cup \{0\}$ does not necessarily exist.

If $W$ is a vector space over the field $K$ and $M\subset W$, let
    \[
    \st(M)=\{w\in W\colon (K\setminus\{0\}) w\subset M\}.
    \]
Notice that
    \begin{textequation}\label{eq1}
    if $V$ is a subspace of $W$, then $V\subset M\cup\{0\}$ if, and only if, $V\subset\st(M)\cup\{0\}$.
    \end{textequation}%
In particular,
    \begin{equation}\label{eq2}
    \LL(M)=\LL(\st(M)).
    \end{equation}
Recall also (see, e.g., \cite{gamezmunozsanchezseoane2010}) that a family $M\subset W$ is said to be {\em star-like}\/ provided $\st(M)=M$. Properties (\ref{eq1}) and (\ref{eq2}) explain why the assumption that $M$ is star-like appears in many results on lineability.

A simple use of Zorn's lemma shows that any linear subspace $V_0$ of $M \cup \{0\}$ can be extended to a maximal linear subspace $V$ of $M \cup \{0\}$. Therefore, the following concept is well defined.

\begin{definition}[maximal lineability cardinal number]
Let $M$ be any arbitrary subset of a vector space $W$. We define
\[
\mL(M)=\min\{\dim( V)\colon \mbox{$V$ is a maximal linear subspace of $M \cup \{0\}$}\}.
\]

\end{definition}

Although this notion might seem similar to that of maximal-lineability and maximal-spaceability (introduced by Bernal-Gonz\'alez in \cite{bernal2010}) they are, in general, not related.

In any case, (\ref{eq1}) implies that $\mL(M)=\mL(\st(M)).$

\begin{remark}
It is easy to see that $\HL(M)=\mL(M)^+$, where $\HL(M)$ is a homogeneous lineability number defined in \cite{bgLAA}.
(This explains why $\HL$ is always a successor cardinal, as shown in \cite{bgLAA}.) Clearly we have
$$
\HL(M)=\mL(M)^+\leq \LL(M).
$$
The inequality may be strict, as shown in \cite{bgLAA}.
\end{remark}

For $M\subset W$
we will also consider  the following \emph{additivity} number (compare~\cite{bgLAA}), which is a generalization of the notion introduced by T.~Natkaniec in \cites{natkaniec1991,natkaniec1996} and thoroughly studied by the first author \cites{STA,CJ,CM,CN,CR} and F.E.~Jordan \cite{jordan1998} for $V=\real^\real$ (see, also, \cite{B-G2010}):
\[
A(M,W)=\min\bigl(\{|F|\colon F\subset W\ \&\ (\forall w\in W)(w+F\not\subset M)\}\cup\{|W|^+\}\bigr),
\]
where $|F|$ is the cardinality of $F$ and $w+F=\{w+f\colon f\in F\}$.
Most of the times
the space \(W\), usually  \(W=\real^\real\), will be clear by the context. In such cases we will often write \(A(M)\)
in place of $A(M,W)$.

We are mostly interested in the topological vector spaces $W$. We say that $M\subset W$ is \emph{$\mu$-spaceable} with respect to a topology $\tau$ on $W$, provided there exists a $\tau$-closed vector space $V\subset M\cup\{0\}$ of dimension $\mu$. In particular, we can consider also the following \emph{spaceability}\/ cardinal number:
\[
\LL_\tau(M)=\min\{\kappa\colon \mbox{$M \cup \{0\}$ contains no $\tau$-closed subspace of dimension $\kappa$}\}.
\]
Notice that $\LL(M)=\LL_\tau(M)$ when $\tau$ is the discrete topology.

In what follows, we shall focus on spaces $W=\real^X$ of all functions from  $X=\real^n$ to $\real$
and consider the topologies $\tau_u$ and $\tau_p$
of uniform and pointwise convergence, respectively.
In particular,
we write $\LL_u(M)$ and $\LL_p(M)$ for $\LL_{\tau_u}(M)$ and $\LL_{\tau_p}(M)$, respectively.
Clearly
\[
\LL_p(M)\leq \LL_u(M)\leq \LL(M).
\]

Recall also a series of definitions that shall be needed throughout the paper.

\begin{definition}
For $X\subseteq\real^n$ a function $f\colon X\to\real$ is said to be
\begin{itemize}

\item \emph{Darboux\/} if $f[K]$ is a connected subset of $\real$ (i.e., an interval) for every connected subset $K$
of $X$;

\item \emph{Darboux\/} in the sense of Pawlak if $f[L]$ is a connected subset of $\real$ for every arc $L$ of $X$ (i.e., $f$ maps path connected sets into connected sets);

\item \emph{almost continuous\/} (in the sense of Stallings) if each open subset of $X\times\real$ containing the graph of $f$ contains also a continuous function from $X$ to~$\real$;

\item a \emph{connectivity\/} function if the graph of $f\restriction Z$ is connected in $Z\times\real$ for any connected subset $Z$ of~$X$;

\item \emph{extendable\/} provided that there exists a connectivity function $F\colon X\times[0,1]\to\real$  such that $f(x)=F(x,0)$ for every $x\in X$;

\item \emph{peripherally continuous\/} if for every $x\in X$ and for all pairs of open sets $U$ and $V$ containing $x$ and $f(x)$, respectively, there exists an open subset $W$ of $U$ such that $x \in W$ and $f[\bd (W)]\subset V$.
\end{itemize}

The above classes of functions are denoted by $\D(X)$, $\Dp(X)$, $\AC(X)$, $\Conn(X)$, $\Ext(X)$, and $\PC(X)$, respectively.
The class of continuous functions from $X$ into~$\real$ is denoted by $\CC(X)$. We will drop the domain~$X$ if $X=\real$.
\end{definition}

\begin{definition}
A function $f\colon \real^n\to\real$ is called
\begin{itemize}
\item \emph{everywhere surjective\/} if \(f[G]=\real\) for every nonempty open set \(G\subset\real^n\);
\item \emph{strongly everywhere surjective\/} if \(f^{-1}(y)\cap G\) has cardinality \(\continuum\) for every \(y\in\real\) and every nonempty open set \(G\subset\real^n\);
this class was also studied in \cite{CM}, under the name
of $\continuum$ strongly Darboux functions;

\item \emph{perfectly everywhere surjective\/} if \(f[P]=\real\) for every perfect set \(P\subset\real^n\)
(i.e., when $f^{-1}(r)$ is a Bernstein set for every $r\in\real$ (compare \cite{Ci1997}*{chap.~7}));
\item a \emph{Jones function\/} (see \cite{jones1942}) if \(f\cap F\neq\varnothing\) for every closed set \(F\subset\real^n\times\real\) whose projection on \(\real^n\) is uncountable.
\end{itemize}
The classes of these functions are written as \(\ES(\real^n)\), \(\SES(\real^n)\), \(\PES(\real^n)\),\linebreak and~\(\J(\real^n)\), respectively. We will drop the domain $\real^n$ if $n=1$.
\end{definition}

\begin{definition}
A function $f\colon\real\to\real$ has:
\begin{itemize}
\item the \emph{Cantor intermediate value property\/}
if for every $x,y\in\real$ and for each perfect set~$K$ between $f(x)$ and $f(y)$
there is a perfect set~$C$ between $x$ and $y$ such that $f[C]\subset K$;

\item
the \emph{strong Cantor intermediate value property\/} if
for every $x,y\in\real$ and for each perfect set $K$ between $f(x)$ and $f(y)$
 there is a perfect set $C$ between $x$ and $y$ such that
 $f[C]\subset K$ and $f\restriction C$ is continuous;

\item the \emph{weak Cantor intermediate value property\/} if for every $x,y\in\real$
with $f(x)<f(y)$ there exists a perfect set $C$ between
$x$ and $y$ such that $f[C]\subset (f(x),f(y))$;

\item \emph{perfect roads\/} if for every $x\in\real$ there exists a perfect set $P\subset\real$ having $x$ as a bilateral (i.e., two sided) limit point for which $f\restriction P$ is continuous at~$x$.
\end{itemize}
The above classes of functions shall be denoted by $\CIVP$, $\SCIVP$, $\WCIVP$, and $\PR$, respectively.
\end{definition}

Notice that all classes defined in the above three definitions are star-like.

The text is organized as follows. In Section~\ref{sec2} we study the relations between additivity and maximal lineability numbers. Sections~\ref{max_extend_R}  and~\ref{max_extend_Rn}  focus on the set of extendable functions on $\real$ and $\real^n$, respectively. Surprisingly enough, we shall obtain very different results when moving from $\real$ to $\real^n$. The lineability of some of the above functions have been recently partly studied (see, e.g., \cite{bgLAA,gamez2011,B-G2010,gamezmunozsanchezseoane2010}) but here we shall give definitive answers
concerning the lineability and spaceability of several previous studied classes.

\section{Relation between additivity and lineability numbers}\label{sec2}

The goal of this section is to examine possible values of numbers  $A(M)$, $\mL(M)$, and $\LL(M)$ for a subset $M$ of a linear space $W$ over an arbitrary field $K$. We will concentrate on the cases when $\varnothing\neq M\subsetneq W$, since it is easy for the cases $M\in\{\varnothing,W\}$. Indeed, as it can be easily checked, one has 
$A(\varnothing)=\LL(\varnothing)=1$ and $\mL(\varnothing)=0$;
$A(W)=|W|^+$, $\LL(W)=\dim(W)^+$, and $\mL(W)=\dim(W)$.

\begin{proposition}\label{prop21}
Let $W$ be a vector space over a field \(K\) and let $\varnothing\neq M\subsetneq\nolinebreak W$. Then
\begin{enumerate}
\item\label{prop1_ml_l} $2\leq A(M)\leq |W|$ and $\mL(M)<\LL(M)\leq \dim(W)^+$;
\item\label{prop1_A_ml} if $A(\st(M))>|K|$, then $A(\st(M))\le\mL(M)$.
\end{enumerate}
In particular, if $M$ is star-like, then $A(M)>|K|$ implies that
\begin{enumerate}\setcounter{enumi}2
\item $A(M)\le\mL(M)<\LL(M)\leq \dim(W)^+$.
\end{enumerate}
\end{proposition}

\begin{proof}
\ref{prop1_ml_l}\quad These inequalities are easy to see.

\ref{prop1_A_ml} This can be proved by an easy transfinite induction. Alternatively, notice that  A.\ Bartoszewicz and S.\ G\l\c{a}b proved,
in \cite{bgLAA}*{corollary 2.3}, that if $M\subset W$ is star-like
and $A(M)>|K|$, then $A(M)<\HL(M)$.
Hence,
$A(\st(M))>|K|$ implies that $A(\st(M))<\HL(\st(M))=\mL(\st(M))^+=\mL(M)^+$.
Therefore,
$A(\st(M))\le\mL(M)$.
\end{proof}

In what follows, we will restrict our attention to the star-like families, since,
by Proposition~\ref{prop21}, other cases could be reduced to this situation.
Our next theorem shows that, for such families and under assumption that $A(M)>|K|$,
the inequalities (3) constitute all that can be said on these numbers.

\begin{theorem}\label{EX21}
Let $W$ be an infinite dimensional vector space over an infinite field $K$ and let $\alpha$, $\mu$, and $\lambda$ be the cardinal numbers such that
$|K|<\alpha\le\mu<\lambda\leq \dim(W)^+$.
Then there exists a star-like $M\subsetneq W$ containing $0$ such that
$A(M)=\alpha$, $\mL(M)=\mu$, and $\LL(M)=\lambda$.
\end{theorem}

The proof of this theorem will be based on the following two lemmas. The first of them
shows that the theorem holds when $\alpha=\mu$, while the second
shows how such an example can be modified to the general case.

\begin{lemma}\label{lemOplus1}
Let $W$ be an infinite dimensional vector space over an infinite field  $K$ and let $\mu$ and~$\lambda$ be the cardinal numbers such that
$|K|<\mu<\lambda\leq \dim(W)^+$.
Then there exists a star-like $M\subsetneq W$ containing $0$ such that
$A(M)=\mL(M)=\mu$ and $\LL(M)=\lambda$.
\end{lemma}

\begin{proof} 
For $S\subset W$, let $V(S)$ be the vector
subspace of $W$ spanned by $S$.

Let $B$ be a basis for $W$. For $w\in W$, let $\supp(w)$ be the smallest subset $S$ of $B$ with $w\in V(S)$
and let $c_w\colon \supp(w)\to K$ be such that $w=\sum_{b\in \supp(w)}c_w(b) b$.
Let $E$ be the set of all cardinal numbers less than~$\lambda$
and choose a sequence $\la B_\eta\colon \eta\in E\ra$ of pairwise disjoint subsets of $B$
such that $|B_0|=\mu$ and
$|B_\eta|=\eta$ whenever 
$0\neq\eta\in E$.
Define
\[
M=\A\cup\smash[b]{\bigcup_{\eta\in E}}V(B_\eta),
\]
where
    \begin{align*}
    \A=\{&w\in W\colon\\ 
    &c_w(b_0)=c_w(b_1) \text{ for some
    $b_0\in \supp(w)\cap B_0$, $b_1\in \supp(w)\setminus B_0$}\}.
    \end{align*}
We will show that $M$ is as desired.

Clearly, $M$ is star-like and $0\in M\subsetneq W$.
Also,  $\LL(M)\geq\lambda$, since for any cardinal $\eta<\lambda$
the set $M$ contains a vector subspace $V(B_\eta)$ with $\dim(V(B_\eta))\geq\eta$.

To see that
$A(M)\geq\mu$, choose an $F\subset W$ with $|F|<\mu$.
It is enough to show that $|F|< A(M)$, that is, that there exists a $w\in W$ with $w+F\subset \A$.
As $\supp(F)=\bigcup_{v\in F}\supp(v)$ has cardinality at most $|F|+\omega<\mu=|B_0|<\lambda\le|B\setminus B_0|$,
there exist $b_0\in B_0\setminus \supp(F)$ and $b_1\in B\setminus(B_0\cup\supp(F))$.
Let $w=b_0+b_1$ and notice that
$w+F\subset\A\subset M$, since for every $v\in F$
we have $b_0\in\supp(w+v)\cap B_0$,
$b_1\in\supp(w+v)\setminus B_0$, and  $c_{w+v}(b_0)=1=c_{w+v}(b_1)$.

Next notice that the inequalities $|K|<\mu\leq A(M)$ and Proposition~\ref{prop21}
imply that $\mu\leq A(M)\leq\mL(M)$. Thus, to finish the proof, it is enough to show that
$\mL(M)\leq\mu$ and $\LL(M)\leq\lambda$.

To see that $\mL(M)\leq\mu$, it is enough to show that $V(B_0)$
is a maximal vector subspace of $M$.
Indeed, if $V$ is a vector subspace of  $W$ properly containing $V(B_0)$, then
there exists a non-zero $v\in V\cap V(B\setminus B_0)$. Choose a $b_0\in B_0$ and
a non-zero $c\in K\setminus\{c_v(b)\colon b\in\supp(v)\}$. Then  $c b_0+v\in V\setminus M$. So,
$V(B_0)$
is a maximal vector subspace of $M$ and indeed $\mL(M)\leq\dim(V(B_0))=\kappa$.

To see that $\LL(M)\leq\lambda$, choose a vector subspace $V$ of $W$ of dimension $\lambda$.
It is enough to show that $V\setminus M\neq\varnothing$.
To see this, for every ordinal $\eta\leq\lambda$ let us define
$\hat B_\eta=\bigcup\{B_\zeta\colon \zeta\in E\cap \eta\}$. Notice that
\begin{equation*}
\text{for every $\eta<\lambda$ there is a non-zero $w\in V$ with $\supp(w)\cap \hat B_\eta=\varnothing$.}
\end{equation*}
Indeed, if $\pi_\eta\colon W=V(\hat B_\eta)\oplus V(B\setminus\hat B_\eta)\to V(\hat B_\eta)$ is the natural projection, then
there exist distinct $w_1,w_2\in V$ with $\pi_\eta(w_1)=\pi_\eta(w_2)$ (as $|V(\hat B_\eta)|<\lambda=\dim(V)$).
Then $w=w_1-w_2$ is as required.

Now, choose a non-zero $w_1\in V$ with $\supp(w_1)\cap B_0=\supp(w_1)\cap \hat B_1=\varnothing$.
Then, $w_1\notin\A$ and
if $\supp(w_1)\not\subset \hat B_\lambda=\bigcup_{\eta\in E}B_\eta$, then also $w_1\notin \bigcup_{\eta\in E}V(B_\eta)$,
and we have $w_1\in V\setminus M$.
Therefore, we can assume that  $\supp(w_1)\subset \hat B_\lambda=\bigcup_{\eta<\lambda}\hat B_\eta$.
Let $\eta<\lambda$ be such that $\supp(w_1)\subset \hat B_\eta$
and choose a non-zero $w_2\in V$ with $\supp(w_2)\cap \hat B_\eta=\varnothing$.
Then $w=w_2-w_1\in V\setminus M$ (since $w\notin \A$, being non-zero with $\supp(w)\cap B_0=\varnothing$,
and $w\notin \bigcup_{\eta\in E}V(B_\eta)$, as its support intersects two different $B_\eta$).
\end{proof}

\begin{lemma}\label{lemOplus}
Let \(W\), \(W_0\), and \(W_1\) be the vector spaces over an infinite field \(K\) such that $W=W_0\oplus W_1$.
Let $M\subsetneq W_0$ and 
    \[
    \F=M + W_1=\{m+W\colon m\in M\ \&\ w\in W_1\}.
    \]
Then
\begin{enumerate}
\item\label{lemoplus_star} If \(M\) is star-like, then \(\F\) is also star-like.
\item\label{lemoplus_add} $A(\F,W)=A(M,W_0)$.
\item\label{lemoplus_ml} If \(0\in M\), then $\mL(\F)=\mL(M)+\dim(W_1)$.
\item\label{lemoplus_ll} If \(0\in M\) and $\dim(W_1)<\LL(M)$, then $\LL(\F)=\LL(M)+\dim(W_1)$.
\end{enumerate}
\end{lemma}

\begin{proof}
In the following, let \(\pi_0\colon W=W_0\oplus W_1\to W_0\) be the canonical projection.

\ref{lemoplus_star}\quad Let \(x\in\F\) and \(\lambda\in K\setminus\{0\}\). Since \(M\) is star-like and \(\pi_0(x)\in M\), we have that \(\pi_0(\lambda x)=\lambda\pi_0(x)\in M\), and hence \(\lambda x\in M+W_1=\F\).

\ref{lemoplus_add}\quad Let us see that $A(M,W_0)\leq A(\F,W)$. To this end, let \(\kappa<A(M,W_0)\).
We need to prove that \(\kappa<A(\F,W)\). Indeed, if \(F\subset W\) and \(|F|=\kappa\), then \(|\pi_0[F]|\le|F|=\kappa\).
So, there exists a \(w_0\in W_0\) such that \(\pi_0[w_0+F]=w_0+\pi_0[F]\subset M\), that is, \(w_0+F\subset M+W_1=\F\). Therefore, \(\kappa<A(\F,W)\).

To see that $A(\F,W)\leq A(M,W_0)$ let $\kappa<A(\F,W)$.  We need to show that $\kappa<A(M,W_0)$.
Indeed, let $F\subset W_0$ be such that $|F|=\kappa$.
Since $|F|<A(\F,W)$, there is a $w\in W$ with $w+F\subset \F$. Then \(\pi_0(w)\in W_0\)
and \(\pi_0(w)+F=\pi_0[w+F]\subset\pi_0[\F]=M\), so indeed $\kappa<A(M)$.

\ref{lemoplus_ml}\quad First notice that it is enough to show that
\begin{textequation}\label{maximal_is_sum}
$V$ is a maximal vector subspace of $\F$ if, and only if, $V=V_0+W_1$, where
$V_0$ is a maximal vector subspace of $M$.
\end{textequation}
Indeed, if $V$ is a maximal vector subspace of $\F$ with $\mL(\F)=\dim(V)$, then, by \eqref{maximal_is_sum},
$\mL(\F)=\dim(V)=\dim(V_0)+\dim(W_1)\geq \mL(M)+\dim(W_1)$.
Conversely, if $V_0$ is a maximal vector subspace of $M$ with $\mL(M)=\dim(V_0)$, then we have
$\mL(M)+\dim(W_1)=\dim(V_0)+\dim(W_1)=\dim(V_0+W_1)\geq \mL(\F)$.

To see \eqref{maximal_is_sum}, take a maximal vector subspace $V$ of $\F$.
Notice that $W_1\subset V$, since
$V\subset V+W_1\subset \F+W_0=\F$ and so, by the maximality, $V+W_1=V$.
In particular, $V=V_0+W_1\subset\F=M+W_1$, where $V_0=\pi_0[V]$.
Thus, $V_0$ is a vector subspace of $M$. It must be maximal, since for any its proper extension $\hat V_0\subset M$,
the vector space
$\hat V_0+W_1\subset\F$ would be a proper extension of $V$.

Conversely, if $V_0$ is a maximal vector subspace of $M$, then $V=V_0+W_1$ is a vector subspace of $\F$.
If cannot have a proper extension $\hat V\subset \F$, since then the vector space $\pi_0[\hat V]\subset M$
would be a proper extension of $V_0$.

\ref{lemoplus_ll}\quad To see that $\LL(\F)\leq\dim(W_1)+\LL(M)$, choose a vector space $V\subset\F$.
We need to show that $\dim(V)<\dim(W_1)+\LL(M)$.
Indeed, $V_1=V+W_1$ is a vector subspace of $\F+W_1=\F$ and
$\dim(V)\leq\dim(V_1)=\dim(W_1)+\dim(\pi_0[V_1])$, since
$V_1=W_1\oplus \pi_0[V_1]$.
Therefore, $\dim(V)\leq\dim(W_1)+\dim(\pi_0[V_1])<\dim(W_1)+\LL(M)$,
since $\dim(W_1)<\LL(M)$
and
$\dim(\pi_0[V_1])<\LL(M)$, as $\pi_0[V_1]$ is a vector subspace of $M=\pi_0[\F]$.
So, $\LL(\F)\leq\dim(W_1)+\LL(M)$.

To see that $\dim(W_1)+\LL(M)\leq\LL(\F)$,
choose a $\kappa<\dim(W_1)+\LL(M)$. We need to show that $\kappa<\LL(\F)$,
that is, that there exists a vector subspace $V$ of~$\F$ with $\dim(V)\geq\kappa$.
First, notice that $\dim(W_1)<\LL(M)$ and $\kappa<\dim(W_1)+\LL(M)$
imply that there exists a $\mu<\LL(M)$ such that $\kappa\leq\dim(W_1)+\mu<\dim(W_1)+\LL(M)$.
(For finite value of $\LL(M)$, take $\mu=\max\{\kappa-\dim(W_1),0\}$;
for infinite $\LL(M)$, the number $\mu=\max\{\kappa,\dim(W_1)\}$ works.)
Choose a vector subspace $V_0$ of $M$ with $\dim(V_0)\geq\mu$.
Then the vector subspace $V=V_0+W_1=V_0\oplus W_1$ of $\F$ is as desired,
since we have $\dim(V)=\dim(W_1)+\dim(V_0)\geq\dim(W_1)+\mu\geq\kappa$.
\end{proof}

\noindent
{\em Proof of Theorem~\ref{EX21}.}
Represent $W$ as $W_0\oplus W_1$, where $\dim(W_0)=\lambda$ and $\dim(W_1)=\mu$.
Use Lemma~\ref{lemOplus1} to find
a star-like $M\subsetneq W_0$ containing $0$ such that
$A(M,W_0)=\mL(M)=\alpha$ and $\LL(M)=\lambda$.
Let $\F=M+W_1\subsetneq B$. Then, by Lemma~\ref{lemOplus},
$\F\ni 0$ is star-like such that
$A(\F)=A(M,W_0)=\alpha$, $\mL(\F)=\mL(M)+\dim(W_1)=\alpha+\mu=\mu$,
and $\LL(\F)=\LL(M)+\dim(W_2)=\lambda+\alpha=\lambda$,
as required. \qed

\medskip

A.\ Bartoszewicz and S.\ G\l\c{a}b have asked
\cite{bgLAA}*{open question 1}
whether the inequality $A(\F)^+\geq \HL(\F)$
(which is equivalent to $A(\F)\geq \mL(\F)$)
holds for any family $\F\subset\real^\real$.
Of course, for the star-like
families $\F$ with $A(\F)>\continuum$, a positive answer to this question
would mean that, under these assumptions, we have $A(\F)=\mL(\F)$.
Notice that Theorem~\ref{EX21} gives, in particular,
a negative answer to this question.

We do not have a comprehensive example, similar to that provided by Theorem~\ref{EX21},
for the case when $A(M)\leq|K|$.
However, the machinery built above, together with the results from \cite{bgLAA},
lead to the following result.

\begin{theorem}
Let \(W\) a vector space over an infinite field~\(K\) with \(\dim(W)\ge 2^{|K|}\).
If $2\le\kappa\le |W|$, there exists a star-like family $\F\subsetneq W$ containing $0$
such that $A(\F)=\kappa$ and $\mL(\F)=\dim(W)$ (so that $\LL(\F)=\dim(W)^+$).
\end{theorem}

\begin{proof}
Represent $W$ as \(W=W_0\oplus W_1\),
where \(\dim(W_0)=2^{|K|}\)
and \(\dim(W_1)=\dim(W)\).
If \(2\le\kappa\le|K|\), then, by \cite{bgLAA}*{Theorem~2.5}, there exists a star-like family \(M\subset W_0\) such that \(A(M,W_0)=\kappa\). Notice that the family constructed in that result contains \(0\).
Then, by Lemma~\ref{lemOplus}, the family \(\F=M+W_1\) satisfies that \(A(\F)=A(M,W_0)=\kappa\)
and $\mL(\F)=\mL(M)+\dim(W_1)=\dim(W)$.
\end{proof}

\section{Spaceability of  Darboux-like functions on $\real$\label{max_extend_R}}

Recall (see, e.g., \cite{CJ}*{chart~1} or \cite{STA}) that we have the following strict inclusions, indicated by the arrows, between the Darboux-like functions from~$\real$
to $\real$.
The next theorem, strengthening the results presented in the table from \cite{Surv}*{page~14},
determines fully the lineability, $\LL$, and spaceability, $\LL_p$, numbers for these classes.

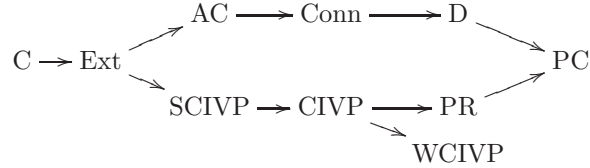
\begin{figure}[h]
\[\xymatrix @R=.4pc @C=1pc
{
               &                          & \AC\ar[r]      &\Conn\ar[r]          &\D\ar[dr]      \\
\CC\ar[r] & \Ext\ar[ur]\ar[dr]&                     &                           &                      &\PC\\
              &                           &\SCIVP\ar[r] &\CIVP\ar[r]\ar[dr] &\PR\ar[ur]    \\
              &                           &                    &                           &\WCIVP
}\]
\bigskip
\caption{Relations between the Darboux-like classes of functions from  \(\real\) to  \(\real\).
Arrows indicate strict inclusions.}\label{Fig1}
\end{figure}

\begin{theorem}\label{th1}
$\LL_p(\Ext)=\left(2^{\continuum}\right)^+$. In particular, all Darboux-like classes of functions from Figure~\ref{Fig1},
except $\CC$, are $2^{\continuum}$-spaceable with respect to the topology of pointwise convergence.
\end{theorem}

\begin{proof}
In \cite{CR}*{corollary 3.4} it is shown that there exists an $f\in \Ext$ and
an $F_\sigma$ first category set $M\subset\real$ such that
    \begin{equation}\label{ext_restricted}
    \text{if $g\in\real^\real$ and $g\restriction M=f\restriction M$, then $g\in \Ext$.}
    \end{equation}
It is easy to see that for any real number $r\neq 0$ the function $r f$ satisfies the same property.

Notice also that there exists a family $\{\,h_\xi\in\real^\real\colon \xi<\continuum\,\}$ of increasing homeomorphisms such that the sets $M_\xi=h_\xi[M]$, $\xi<\continuum$, are pairwise disjoint.
(See, e.g., \cite{CR}*{lemma 3.2}.) It is easy to see that each function $f_\xi=f\circ h_\xi^{-1}$ satisfies~\eqref{ext_restricted} with the set $M_\xi$. Increasing one of the sets $M_\xi$, if necessary, we can also assume that
$\{M_\xi\colon \xi<\continuum\}$ is a partition of $\real$.
Let $\vec f=\la f_\xi\restriction M_\xi\colon \xi<\continuum\ra$ and
define
\begin{equation}\label{spaceV}
V(\vec f\,)=\biggl\{\bigcup_{\xi<\continuum} t(\xi) (f_\xi\restriction M_\xi)\colon t\in\real^\continuum\biggr\}.
\end{equation}
%
It is easy to see that $V(\vec f\,)$ is $2^\continuum$-dimensional $\tau_p$-closed linear subspace\linebreak of~$\Ext$.
\end{proof}

As the cardinality of
the family \(\bor\) of Borel functions from $\real$ to $\real$
is \(\continuum\), Theorem~\ref{th1} easily implies  that \(\Ext\setminus\bor\) is \(2^\continuum\)-lineable:
$\LL(\Ext\setminus\bor)=\left(2^{\continuum}\right)^+$. Actually, we have an even stronger result:

\begin{proposition}\label{prop1}
$\LL_p(\Ext\cap \SES\setminus \bor)=\left(2^{\continuum}\right)^+$.
\end{proposition}

\begin{proof}
The function $f\restriction M$ satisfying \eqref{ext_restricted} may also have the property that
    \begin{equation}\label{ext_restricted2}
    \text{$M$ is $\continuum$-dense in $\real$ and $f\restriction M$ is $\SES$ non-Borel.}
    \end{equation}
Indeed, this can be ensured by enlarging $M$ by a $\continuum$-dense first category set $N\subset\real\setminus M$
and redefining $f$ on $N$ so that $f\restriction N$ is non-Borel and $\SES$.

Now, if $f$ satisfies both \eqref{ext_restricted} and \eqref{ext_restricted2} and $\vec f=\la f_\xi\restriction M_\xi\colon \xi<\continuum\ra$
is defined as in Theorem~\ref{th1}, then the space $V(\vec f\,)$ given in \eqref{spaceV} is as required.
\end{proof}

Notice also that $\Ext\cap \PES=\PR\cap\PES=\varnothing$. In particular, the space
$V$ from Proposition \ref{prop1} is disjoint with $\PES$.

\begin{remark}
Clearly, Theorem~\ref{th1} implies that $\Ext$ is $2^\continuum$-lineable. This result has been also independently proved by T. Natkaniec. (See preprint~\cite{NatkLAA}.) The technique used in \cite{NatkLAA} is similar, but different from that used in the proof of Theorem \ref{th1}.
\end{remark}

Recall, that it is known that $\LL(\AC\setminus \Ext)=\left(2^{\continuum}\right)^+$. See \cite{gamezmunozsanchezseoane2010} or \cite{Surv}*{page~14}. However, we do not know what the exact values of the following cardinals are.

\begin{problem}
Determine the following numbers:
$$\LL_p(\F\setminus\G), \, \LL_u(\F\setminus\G), \text{ and }\LL_u(\F\setminus\G)$$ for
$\F\in\{\Conn\setminus\AC,\D\setminus\Conn,\PC\setminus\D\}$ and $\G\in\{\SCIVP,\CIVP,\PR\}$.
\end{problem}

\begin{problem}
Is it consistent with the axioms of set theory ZFC that either $A(\F)<\mL(\F)$ or $\mL(\F)^+<\LL(\F)$
for any of the classes $\F\in\{\Ext,\AC,\Conn,\D,\PC\}$?
\end{problem}

Notice, that the generalized continuum hypothesis GCH implies that
$A(\F)=\mL(\F)$ and $\mL(\F)^+=\LL(\F)$ for every $\F\in\{\Ext,\AC,\Conn,\D,\PC\}$.

\section{Spaceability of Darboux-like functions on $\real^n$, $n\geq 2$\label{max_extend_Rn}}

Recall (see, e.g., \cite{CJ}*{chart 2} or \cite{STA}) that we have the following strict inclusions, indicated by the arrows,
between the Darboux-like functions from~$\real^n$ to $\real$ for $n\geq 2$.

\begin{figure}[h]
\[\xymatrix @R=.4pc @C=.7pc {&\Conn(\real^n)\ar@{=}[d]&&\AC(\real^n)\\
\CC(\real^n)\ar[r]&\Ext(\real^n)\ar@{=}[d]\ar[r]&\AC(\real^n)\cap\D(\real^n)\ar[ru]\ar[rd]&\\
&\PC(\real^n)&&\D(\real^n)}\]
\caption{Relations between the Darboux-like classes of functions from  \(\real^n\) to  \(\real\), \(n\ge2\).
Arrows indicate strict inclusions.}\label{fig2}
\end{figure}

The proof of the next theorem will be based on the following result \cite{CW}*{Proposition\ 2.7}:

\begin{proposition}\label{Proplca}
Let $n>0$ and let
$f\colon \real^n\rightarrow \real$ be a
peripherally continuous function. Then for any $x_0\in \real^n$ and
any open set $W$ in $\real^n$ containing $x_0$,
there exists an open set
$U\subseteq W$ such that $x_0\in U$ and the restriction of $f$ to
$\bd(U)$ is continuous.
Moreover, given any $\varepsilon>0$, the set $U$ can be chosen
so that $\left|f(x_0)-f(y)\right|<\varepsilon$ for every $y\in \bd(U)$.
\end{proposition}

\begin{theorem}\label{th2}
For $n\geq 2$, $\LL_p(\Ext(\real^n))=\LL_u(\Ext(\real^n))=\LL(\Ext(\real^n))=\continuum^+$.
In particular,
the classes $\CC(\real^n)$ and $\Ext(\real^n)$ are ${\continuum}$-spaceable
with respect to the pointwise convergence topology $\tau_p$
but are not $\continuum^+$-lineable.
\end{theorem}

\begin{proof}
First, notice that $\LL_p(\CC(\real^n))=\continuum^+$ is justified by the space $\CC_0$
of all continuous functions linear on the interval $[k,k+1]$ for every integer $k\in\integer$.
Indeed, $\CC_0$ is linearly isomorphic to $\real^\integer$.

Now, since
$\continuum^+=\LL_p(\CC(\real^n))\leq\LL_p(\Ext(\real^n))\leq\LL_u(\Ext(\real^n))\leq\LL(\Ext(\real^n))$,
it is enough to show that $\LL(\Ext(\real^n))\leq\continuum^+$, that is, that
$\Ext(\real^n)$
is not $\continuum^+$-lineable.
To see this,
by way of contradiction, assume that there exists a vector space $V \subset \Ext(\real^n)$ of cardinality
greater than $\continuum$.
Fix a countable dense set $D\subset\real^n$ and let
$\la \la x_k,\e_k\ra\colon k<\omega\ra$ be an enumeration of $D\times\{2^{-m}\colon m<\omega\}$.
By Proposition~\ref{Proplca},
for every function $f\in \Ext(\real^n)$ and $k<\omega$ we can choose an open neighborhood $U_k^f$ of $x_k$
of the diameter at most $\e_k$ such that $f\restriction \bd(U_k^f)$ is continuous.
Consider the mapping
$V\ni f\mapsto T_f=\la f\restriction \bd(U_k^f)\colon k<\omega\ra$.
Since its range has cardinality $\continuum$, there are distinct
$f_1,f_2\in V$ with $T_{f_1}=T_{f_2}$.
In particular, $f=f_1-f_2\in V$ is equal zero on the set $M=\bigcup_{k<\omega} \bd(U_k^{f_1})$.
Notice that the complement $M^c$ of $M$ is zero-dimensional.
We will show that $f$ is not extendable, by showing that it does not satisfy Proposition~\ref{Proplca}.

Indeed, since $f_1\neq f_2$, there is an $x\in\real^n$ with $f(x)\neq 0$.
Let $\e=|f(x)|$ and let $W$ be any bounded neighborhood of $x$.
Then, there is no set $U$ as required by Proposition~\ref{Proplca}.

To see this, notice that for any open set $U\subseteq W$ with $x\in U$,
its boundary is of dimension at least 1. In particular,
$M\cap \bd(U)\neq \varnothing$ and, for $y\in M\cap \bd(U)$, we have
$\left|f(x)-f(y)\right|=|f(x)|=\varepsilon$.
\end{proof}

Theorem~\ref{th2} determines the values of the numbers
$\LL_p(\F)$, $\LL_u(\F)$, and $\LL(\F)$ for
$\F\in\{\CC(\real^n),\Ext(\real^n),\Conn(\real^n),\PR(\real^n)\}$ and $n\geq 2$.
In the remainder of this section we will
examine these cardinal numbers for the remaining classes
from the diagram in Figure~\ref{fig2}.
For this, we will need the following fact,
improving a recent result of the second author, see \cite{gamez2011}*{Theorem 2.2}.


\begin{proposition}\label{jones}
$\LL_p(J(\real^n))=\left(2^\continuum\right)^+$ for every $n\geq 1$.
In particular, the families \(J(\real^n)\),
$\PES(\real^n)$,  $\SES(\real^n)$, and $\ES(\real^n)$ are $2^\continuum$-spaceable
with respect to the topology of pointwise convergence.
\end{proposition}

\begin{proof}
Let $\{\B_\xi\colon\xi<\continuum\}$ be a decomposition of \(\real^n\) into pairwise disjoint Bernstein sets.
For every $\xi<\continuum$, let $f_\xi\colon B_\xi\to\real$ be such that
\(f_\xi\cap F\neq\varnothing\) for every closed set \(F\subset\real^n\times\real\) whose projection on \(\real^n\) is uncountable.
(All of this can be easily constructed by transfinite induction. See, e.g., \cite{Ci1997}.)
Notice that
    \begin{textequation}\notag
    if $g\in\real^\real$ and $g\restriction M_\xi=r\ f_\xi$ for some
    $\xi<\continuum$ and $r\neq 0$, then $g\in J(\real^n)$.%
    \end{textequation}%
Now, if $\vec f=\la f_\xi\restriction M_\xi\colon \xi<\continuum\ra$ and
$V(\vec f\,)$ is given by (\ref{spaceV}), then
$V(\vec f\,)$ is $2^\continuum$-dimensional $\tau_p$-closed linear subspace of $J(\real^n)$.
\end{proof}

Every function in $\J(\real^n)$ is surjective. In particular, the above result
implies that the class of surjective functions is $2^\continuum$-lineable.
One could also wonder about the lineability of the family of one-to-one functions from $\real^n$
to $\real$, given below.

\begin{remark}
The family of one-to-one functions from $\real^n$ to $\real$ is $1$-lineable but not $2$-lineable.
\end{remark}

\begin{proof}
Clearly the family is $1$-lineable. To see that  is not $2$-lineable, choose
two injective linearly independent functions $f$ and $g$ generating a linear space $Z$.
Take arbitrary $x\neq y$ in $\real^n$
and consider the function $h = f + \alpha g \in Z \setminus \{0\},$ where $\alpha = (f(x)-f(y))/(g(y)-g(x)) \in \real$.
Then, we have $h(x) = h(y)$, so $Z$ contains a function which is not one-to-one.
\end{proof}

Other examples of $1$-lineable but not $2$-lineable sets and, in general, not lineable sets can be found in \cite{Surv,QJM}.

\begin{theorem}
For $n\ge2$, \(\J(\real^n)\subset \AC(\real^n)\setminus\D(\real^n)\).
In particular, the class $\AC(\real^n)\setminus \D(\real^n)$ is $2^\continuum$-spaceable and
$\LL_p(\AC(\real^n)\setminus \D(\real^n))=\left(2^\continuum\right)^+$.
\end{theorem}

\begin{proof}
By Proposition~\ref{jones}, it is enough to show that \(\J(\real^n)\subset \AC(\real^n)\setminus\D(\real^n)\).
Clearly, \(\J(\real^n)\subset\AC(\real^n)\cap\PES(\real^n)\) for any $n\geq 1$.
Thus, it is enough to show that $\PES(\real^n)\cap\D(\real^n)=\varnothing$ for $n\ge2$.
But this follows immediately from the fact that, under $n\ge2$, every Bernstein set in $\real^n$ is
connected.
\end{proof}

\begin{remark}
Notice that, since $\AC(\real^n) \subset \Dp(\real^n)$, then,
for $n\geq 2$, we have $\LL_p(\Dp(\real^n) \setminus \D(\real^n))=\left(2^\continuum\right)^+$.
So, $\Dp(\real^n) \setminus \D(\real^n)$ is also $2^\continuum$-spaceable.
\end{remark}

\begin{theorem}
For $n\geq 2$, $\LL_p(\D(\real^n)\setminus\AC(\real^n))=\left(2^\continuum\right)^+$.
In particular, the class $\D(\real^n)\setminus\AC(\real^n)$ is ${2^\continuum}$-spaceable.
\end{theorem}

\begin{proof}
Let \(\pi_1\colon\real^n\to\real\) the projection of \(\real^n\) on its first coordinate.
Let \(W=V(\vec f\,)\subset\J\) be the vector space of cardinality \(2^\continuum\) build in Proposition~\ref{jones}.
Then the vector space
    \[
    V=\{\,f\circ\pi_1:f\in W\,\}
    \]
is obviously contained in \(\D(\real^n)\) and has dimension \(2^\continuum\). On the other side, if \(f\in W\) then \(f\circ\pi_1\) cannot be in \(\AC(\real^n)\), because then \(f\) would be continuous. (See \cite{lipinski}.) This is not possible, because \(\J\cap\CC=\varnothing\). Therefore, \(V\subset\D(\real^n)\setminus\AC(\real^n)\). To finish, let us remark that the space \(V\) is also closed by pointwise convergence.
\end{proof}

\begin{remark}
Notice that, in $\real^n$ (for every $n \in \mathbb{N}$), we have that $\AC\setminus \Ext$ is $2^{\mathfrak{c}}$-spaceable (since this class contains the Jones functions). Also, in $\real$, $\J \subset \AC \setminus \SCIVP \subset \AC \setminus \Ext$ and, since $\LL_p (\J) = (2^{\mathfrak{c}})^+$, we have (from the previous results) that $$\LL_p(\AC\setminus \Ext) = \LL_u(\AC\setminus \Ext) = (2^{\mathfrak{c}})^+.$$
\end{remark}

\begin{problem} For $n\geq 2$, determine the values of the numbers
$\LL_p(\AC(\real^n)\cap\D(\real^n))$,
$\LL_u(\AC(\real^n)\cap\D(\real^n))$, and
$\LL_p(\AC(\real^n)\cap\D(\real^n))$.
\end{problem}



\end{document}